\documentclass[12pt]{article}
\usepackage{amssymb, latexsym, amsmath, amsthm, amsfonts, amssymb, paralist, extpfeil, mathrsfs}
\usepackage{cases}
\usepackage[title,titletoc]{appendix}
\usepackage{yfonts}
\usepackage[dvipsnames,svgnames,x11names,table]{xcolor}
\usepackage{color}
\usepackage{oplotsymbl}
\usepackage{graphicx,overpic,booktabs,makecell}
\usepackage{hyperref}
\hypersetup{
	CJKbookmarks=true,
	bookmarksopen=true,
	pdfstartview=FitH,
	colorlinks=true,
	linkcolor=blue,
	filecolor=blue,
	urlcolor=blue,
	citecolor=blue,
	bookmarksnumbered=true,
}

\usepackage{geometry}
\usepackage{enumerate}

\usepackage{xifthen}
\usepackage[numbers,sort&compress]{natbib}
\renewcommand{\cite}[2][]{\ifthenelse{\isempty{#1}}{\citep{#2}}{\citetext{\citenum{#2}, #1}}}

\def\ii{{\mathbb{I}}}

\def\nn{{\mathbb N}}

\def\rr{{\mathbb R}}

\def\zz{{\mathbb Z}}

\def\bV{{\mathbf V}}

\def\cB{\mathcal{B}}

\def\cE{\mathcal{E}}

\def\cN{\mathcal{N}}

\def\rB{\mathrm{B}}

\def\az{\alpha}

\def\dz{\delta}
\def\ez{\epsilon}
\def\kz{\kappa}
\def\lz{\lambda}
\def\Lz{\Lambda}

\def\tz{\theta}

\def\gz{\gamma}
\def\Gz{\Gamma}

\def\lsim{\lesssim}

\def\hh{{\mathbb{H}}}

\def\vol{\operatorname{vol}}

\def\sgn{\operatorname{sgn}}

\def\fall{\forall\,}

\def\pt{\partial}

\newtheorem{theorem}{Theorem}[section]

\newtheorem{lemma}[theorem]{Lemma}

\newtheorem{conjecture}[theorem]{Conjecture}

\theoremstyle{definition}

\newtheorem{remark}[theorem]{Remark}

\numberwithin{equation}{section}

\title{\bf Lattice point counting in Cygan--Kor\'anyi balls on Heisenberg groups}
\author{Sheng-Chen Mao\footnote{Corresponding author. \texttt{maoshengchen@lzu.edu.cn}. {\color{red!60!black}
} }, \date{}
\quad Sibei Yang
}

\begin{document}
\renewcommand{\theequation}{\thesection.\arabic{equation}}
\setcounter{equation}{0} \maketitle
\vspace{-1.0cm}
\bigskip
	
{\noindent\bf Abstract.}
Lattice point counting  in gauge balls on  the Heisenberg  group $\mathbb{H}^q$ is a non-commutative analogue of the Euclidean multidimensional  sphere problem, initiated by Garg, Nevo and Taylor  \cite[\textit{Ann. Inst. Fourier}, 2015]{GNT15}. The case of particular interest is when the gauge is taken as the Cygan--Kor\'anyi norm and the error term reads: $$\mathcal{E}_q(t)=\#\left(\mathbb{Z}^{2 q+1} \cap  \mathcal{B}_t\right)-\operatorname{vol}(\mathcal{B}_1) \, t^{2 q+2},$$
with $\mathcal{B}_t=\{(v,w)\in\mathbb{H}^q: (|v|^4 + w^2)^{1/4} \le t \}$, which is closely related to the Gauss circle problem. When $q\ge3$, Gath  \cite[\textit{Ann. Sc. Norm. Super. Pisa Cl. Sci.}, 2022]{Gat22} improved upon   \cite[]{GNT15} by showing that $ |\mathcal{E}_q(t)|\lesssim t^{2q-1+ 1/3}$ and proposed the conjecture that the optimal order should be $2q-1$. In this paper, through Landau's formula and the $5,6$-th Derivative Tests of van der Corput, we arrive at that $|\mathcal{E}_q(t)| \lesssim t^{2 q-1 + 241/753} $ for any $ q \geq 4$, and recover the bound of Gath for $q=3$ up to a logarithmic factor. This, via a simpler method,  provides  the first  progress towards  Gath's conjecture.

	
	
	
\section{Introduction} \label{intro}
The Gauss circle problem is one of the most famous and difficult open problems in analytic number theory, aiming to determine the optimal order of the magnitude for the error term
\begin{equation} \label{gauE} 
    P_2(x) :=  \#\left\{m\in\zz^2: m_1^2 + m_2^2 \le x \right\} - \pi x 
\end{equation}
for large $x$. Up to now, the best upper bound $P_2(x) = O_\ez(x^{\theta^*+\ez}) $ ($\fall \ez>0$) is due to Li and Yang  \cite[]{LY23a}, where $\theta^* = \frac{25 \sqrt{1717}+3292}{13762} \approx 0.314483 $, by combining the Bombieri--Iwaniec method and  some techniques from the decoupling theory in Harmonic Analysis. And the  sharp lower bound $\Omega_{-}\left((x \log x)^{1 / 4}\left(\log _2 x\right)^{3\left(2^{1 / 3}-1\right)/4}\right)$
conjectured by  Soundararajan \cite[]{Sou03} (up to a factor $(\log\log x)^{o(1)}$) and  Lamzouri \cite[]{Lam25} has   been obtained by  Lamzouri \cite[]{Lam26} very  recently. The  optimal  order  is conjectured  to be $1/4$. 

The planar Gauss circle problem admits a natural generalization to higher-dimensional spaces, where the optimal order has been found in dimensions four and above, while the three-dimensional case is  lack of understanding as well,  left with the conjectured optimal order $1/2$; see  \cite[]{IKKN06,BKZ18,Kra88} for  related developments. Lattice point counting in non-Euclidean settings is equally interesting; examples include hyperbolic geometry  \cite[]{BL24,GNY17}, locally compact  groups  \cite[]{GN12,GN10},  etc. The present work focuses primarily on the Heisenberg group, which is widely regarded as  the simplest non-commutative  and non-compact Lie groups.

Let $q$ be a positive integer. Recall the Heisenberg group $\hh^q \cong \rr^{2q} \times \rr$ is defined by the following group law
\begin{equation} \label{Hd} 
(v,w) \circ (v',w') = \left(v + v', w + w' + 2\langle  \mathbf{J} v, v^{\prime} \rangle \right)\quad  \mbox{with}\quad \mathbf{J}=   \left(\begin{array}{cc}
0 & \ii_q \\
-\ii_q & 0
\end{array}\right),
\end{equation}
where $\langle,\rangle $ represents the standard Euclidean inner product.  In such coordinate, the identity element
coincides with $(0,0)$, and the inverse of   $(v,w)\in\hh^q$ equals $(-v,-w)$. The Haar measure on $\hh^q$ is consistent with the usual Lebesgue measure, written as $\vol$. As is well known, the Heisenberg group carries the following anisotropic dilation structure: 
\begin{equation*}
    \delta_r(v, w):=\left(r v, r^2 w\right), \qquad \fall (v, w) \in \hh^q,\, r>0,
\end{equation*}
which constitutes an automorphism group of $\hh^q$. The Cygan--Kor\'anyi norm  (\cite[]{Cyg78,Cyg81,Kor85})
$$\cN(v,w):= (|v|^4 + w^2)^\frac14, \qquad (v,w)\in\hh^q $$ naturally emerges in the expression of the fundamental solution of the corresponding sub-Laplacian on Heisenberg groups, and satisfies the triangle inequality
\begin{equation*}
    \cN((v,w)\cdot(v',w')) \le \cN(v,w) + \cN(v',w'), \quad \fall (v,w),\, (v',w')\in\hh^q,
\end{equation*}
which determines a left-invariant distance $d_\cN$ on $\hh^q$: $$d_\cN((v,w),(v',w')):=\cN((v',w')^{-1}\cdot(v,w)), \quad \fall (v,w),\, (v',w')\in\hh^q .$$ 
This gauge function is $\dz_r$-homogeneous of degree 1, that is,
\begin{equation*}
    \cN( \delta_r(v, w)) = r\, \cN(v,w), \qquad \fall (v, w) \in \hh^q,\, r>0.
\end{equation*}
The interested readers are referred to  \cite[]{BLU07} for more discussions 
on the Heisenberg group and the Cygan--Korányi norm. 

Given any $t>0$,  let 
$$ \mathcal B_t:=\delta_t\,\mathcal B_1
=\bigl\{(v,w)\in\mathbb H^q:\mathcal N(v,w)\le t\bigr\}.$$
We denote the error term of lattice point counting for the Cygan--Kor\'anyi ball $\mathcal{B}_t $ by
\begin{equation} \label{disc} 
    \mathcal{E}_q(t) :=\#\left(\mathbb{Z}^{2 q+1} \cap \delta_t\, \mathcal{B}_1\right)-\operatorname{vol}(\mathcal{B}_1) \, t^{2 q+2}.
\end{equation}
Our goal is  to solve  the optimal order for $\cE_q $, which equals the infimum of all possible $\kz$ such that $|\cE_q(t)|\lsim  t^\kz$ holds.
In 2015, Garg, Nevo and Taylor  \cite[]{GNT15} established the  initial result  as follows:
\begin{equation} \label{gntR} 
   |\cE_q(t)| \lsim 
   \begin{cases}
       t^{2} \log t, &\mbox{as}\ q=1, \\
       t^{4 }  \log^\frac23 t,  &\mbox{as}\ q=2,  \\
       t^{2 q}, &\mbox{as}\ q\ge3. \\
       \end{cases} 
\end{equation}
For $q=1$, the order $2$ in  \eqref{gntR} was shown to be optimal  by Gath  \cite[]{Gat20} in 2020,  using the Mellin transform and some Riesz mean estimates of  Chandrasekharan and Narasimhan  \cite[]{CN61}. Thus, the three-dimensional problem is settled. Later in 2022, for $q\ge3$, Gath \cite[]{Gat22} improved the upper bound  in \eqref{gntR} to $t^{2q-\frac23}$, and derived the lower bound 
\begin{equation*}
    \mathcal{E}_q(t)=\Omega\left(t^{2 q-1}(\log t)^{\frac14}(\log \log t)^{\frac18}\right) .
\end{equation*}
Additionally, in  \cite[]{Gat22} he   proved that $\cE_q(t) $ has a sharp second moment estimate of the magnitude $t^{2q-1}$. Based on these results, one is naturally led to the following conjecture:
\begin{conjecture}[Gath \cite{Gat22}] \label{conj} 
Let $q\ge3$. Then the optimal order of the error term  \eqref{disc} of the Cygan--Kor\'anyi ball problem is $2q-1 $. 
\end{conjecture}
We note that Gath also made this conjecture accordingly for $q=2$ in  \cite[]{Gat22} and the following paper  \cite[]{Gat24a}. Nevertheless, to our best knowledge, unlike the higher-dimensional case, no theoretical support in terms of corresponding moment estimates and lower bound estimates has been supplied, nor has any subsequent public document addressed this issue. The most recent progress is due to the first author  \cite[]{Mao26}, where the upper bound $\cE_2(t)=O(t^{4}) $ is attained, slightly  improving  \eqref{gntR}. 

On the other hand,  the asymptotic behaviour of the error term  of the present $ (2q + 1)$-dimensional  Cygan--Kor\'anyi ball problem with $q \ge2 $ exhibits a distinct nature from the case $q = 1$, and is closely related to the behaviour of that in the Gauss circle problem. Moreover, the case $q=2$ turns out to be a transition between $q=1$ and $q\ge3 $, with somewhat more sophisticated  arithmetic structure.  These  features  reveal a strong similarity between the Cygan--Kor\'anyi ball problem and the Euclidean multidimensional sphere problem, which therefore indicates the resolution of Gath's conjecture ($q\ge2$) should be comparably difficult. For a clearer explanation, cf. also \cite[\S 1.1]{Gat24a}.

With respect to  lattice point counting problems on Heisenberg groups related to more general gauge functions, the reader can consult   \cite[]{Mao26,MY26,GNT15}.

\subsection{Main result}  \label{mainR} 
Our main result provides the first progress towards Gath's conjecture  \eqref{conj}, which reads as follows:
\begin{theorem} \label{mainT} 
Let $q\ge3$ be an integer and $t\ge 10$. Then 
\begin{equation} \label{maine} 
    \left|\cE_q(t)\right| \le {\bf C}_q \begin{cases}t^{2 q-2+\frac{994}{753}}, & q \geq 4, \\ t^{\frac{16}{3}} \log t, & q=3,\end{cases}
\end{equation}
where ${\bf C}_q$ is a constant  depending only on $q$.
\end{theorem}

Here is a remark on this theorem. 
\begin{remark} \label{thRe} 
(1) The exponent $ \frac{994}{753}$ for $q\ge4$ comes from the $6$-th Derivation Test method due to van der Corput. In fact, we have reduced the Cygan--Kor\'anyi ball problem to  an  $\Psi$-sum which connects  the Gauss circle problem in an essential way; see Remark  \ref{reWel} for more details. As a result, one may expect that the method and Li-Yang \cite[]{LY23a}  could be used to further improve our exponent from $\frac{994}{753}$ to $4\tz^*+\ez$. Note that $\frac{994}{753}\approx 1.32005 $ and $4\tz^*\approx 1.25793$.

(2) For $q=3$ the bound is cruder, owing to the outer summation in Landau's formula, where we lose some precision; see also Remark  \ref{reWel} below. But that already reproduces  Gath's previous result with only an $\log t$-loss.

(3) The same argument can be run for $q=2$, but the remainder in Landau's formula already contributes $O(t^4\log t)$. It therefore does not improve the bounds $O(t^4(\log t)^{2/3})$ of \cite{GNT15} or $O(t^4)$ of \cite{Mao26}.

\end{remark}

\subsection{The methods}  \label{outl}
In the earlier work of Garg, Nevo and Taylor \cite{GNT15}, they employed  two strategies for the lattice point counting problem in a larger class of gauge balls with radii $t$:
\begin{equation*}
   \cB^{\az,A}_t := \dz_t\, B^{\az,A}_1 = \{(v,w)\in\hh^q: (|v|^\az + A|w|^{\az/2})^{1/\az}\le t\}
\end{equation*}
for $\az,A>0$. The first one  is classical, by  smoothing out $\chi_{\cB^{\az,A}_1}$ and utilizing the Poisson summation formula as well as the spectral estimates for the ball $ \cB^{\az,A}_1$. For $ \az=4$, as opposed to the Euclidean case, the Gaussian curvature  of sphere $\pt \cB^{\az,A}_1$ can vanish at both the equator $\{w = 0\}$ and the pole $\{v = 0\}$. To remedy this, they applied the partial radial property of the gauges  and then reduced the spectral estimate to an one-dimensional oscillatory problem. Their second method exploits a hyperplane slicing argument with  correlated results of lattice counting  in Euclidean balls, which converts  the original problem   to the estimation of an one-dimensional sum. Based on  \cite[]{GNT15},  the first author \cite[]{Mao26} incorporated recursion formulas of Bessel functions as a new  ingredient,  which is extensively used  to better capture the cancellation of relevant oscillatory integrals. Before that, Gath  \cite[]{Gat22}   developed a more delicate approach, involving a  restricted slicing argument, and some results concerning weighted integer lattice points in suitable Euclidean balls and shrinking annuli.
From   classical tools in analytic number theory, e.g. Vaaler's Lemma and  $B$-process of Karatsuba-Korolev, Gath obtained an initial expression for the error term $\cE_q$ (see  \cite[Proposition 3.5]{Gat22}). This expression is the starting point and plays an important role in the proof of Gath's three theorems. 

Our method is distinct from  \cite[]{GNT15,Gat22,Mao26}. Through a different way of slicing, combined with Landau's formula  \eqref{ff61} for the lattice counting  error term of Euclidean ball problem, we derive a brand  new error term formula for the Cygan--Kor\'anyi ball problem; see  \eqref{Eqt}-\eqref{ff77}. By means of this formula, we   reduce our problem to the estimate of a family of so-called $\Psi$-sums that are 1-periodic, uniformly bounded and of uniformly bounded variation, closely related to the  $\psi$-sum appearing in the Gauss circle problem; see Remark  \ref{reWel} for the detail. Next, with the aid of Ste\u{c}kin's inequality, such $\Psi$-sums can be transformed into the associated exponential sums. Finally, to handle such exponential sums, we employ the $5$-th and $6$-th Derivative Tests of van der Corput. This procedure maybe reminiscent of  Nieland's paper  \cite[]{Nie28} which deals with the Gauss circle problem. By contrast, however, the phase function here we encounter contains an additional parameter $k$, and the summation range to be treated can be much broader, causing a vanishing phenomenon in the $5$-th derivative of the phase. To overcome this difficulty, we resort to the $6$-th Derivative Test.  With the non-vanishing intervals of both $5$-th and $6$-th derivatives, one is able to cover the entire possible summation region, thereby leading to the desired result after some uniform estimates.

We conclude this section by introducing some notation. We  denote the set of natural numbers by $\nn:=\{0,1,2,\ldots\}$, and write $\nn^*=\nn\setminus\{0\}$.   Empty sums are understood to be zero. For $u\in\rr$, we use $[u]$ to signify the largest integer not exceeding $u$, and put $\{u\}:=u-[u]$. Let $\psi(\cdot)$ be the sawtooth function  
$$\psi(u):=u-[u]-\frac12, \quad u\in\rr.$$
We utilize the usual asymptotic notation. For instance, the letter $C$ together with its other variants  will represent  implicit positive constant which may change  between lines. Given a non-negative function $w$.  We say $f = O(w)$ if $|f| \leq C w$; and when $f$ is also non-negative, we shall write $f \leq C w$ (resp. $f \ge C\, w$) for $f \lesssim w$ (resp. $f \gtrsim w$). If both cases happen, we will put $f\sim g$ for short. As the coefficient depends on other parameters, we will indicate this by subscripts, e.g., $O_{\gz},\lesssim_{\gz},\sim_{\gz}$. But we usually suppress the dependence on the dimension.

\section{Proof of the main result}  \label{prf}
\subsection{Landau's formula}  \label{lanS} 
For $r> 0, $ we use $P_d(r)$ to represent the error term of  lattice point counting for the standard closed Euclidean  ball centered at $0$ of radius $\sqrt{r}$ in $\rr^d$, that is, 
\begin{equation*}
    P_d(r) :=\#\left(\zz^d \cap  B_d(\sqrt{r})\right) - \vol(B_d(1)) \, r^{\frac{d}{2}} .
\end{equation*}
When $d\ge4$, Landau \cite[]{Lan24} in 1924 established the following formula (see  \cite[(4)]{Lan24} for a more general result), as  $r\ge 10$,
\begin{equation} \label{land1} 
    \#\left(\zz^d \cap  B_d(\sqrt{r})\right)  =\frac{\pi^{\frac{d}{2}}}{\Gamma\left(\frac{d}{2}\right)} \sum_{1 \leq k \leq \sqrt{r}} \sideset{}{^*}\sum_{j=0}^{k-1} \frac{S_{j, k}}{k^d} \sum_{0 \leq n \leq r} n^{\frac{d}{2}-1} e^{-\frac{2 \pi i j n}{k}}+O_d\left(r^{\frac{d}{4}} \log r\right) ,
\end{equation}
where \begin{footnotesize}$\sideset{}{^*}\sum$\end{footnotesize} indicates summation restricted to those $j$ coprime to $k$ (i.e., $(j,k)$=1), and 
\begin{equation*}
    S_{j, k}:=\sum_{m_1=\cdots=m_d =0}^{k-1} e^{2 \pi i \frac{j}{k}\left(m_1^2+\cdots+m_d^2\right)}=\left(\sum_{l=0}^{k-1} e^{2 \pi i \frac{j}{k} l^2}\right)^d.
\end{equation*}
One has $S_{0,1}=1$. According to  \cite[(12)]{Lan24}, it also holds that
\begin{equation} \label{ff41} 
    \left|S_{j, k}\right| \leq C_d\, k^{\frac{d}{2}}, \qquad \fall 1 \leq j \leq k-1,\ (j, k)=1 .
\end{equation} 
From these, we can obtain a useful expression of the error term $P_d$. 
\begin{lemma}[Landau's formula] \label{lan} 
Suppose  $d\ge4$ and $r\ge10$. Then
\begin{equation} \label{ff61} 
    \begin{aligned}
    P_d(r)= & -\frac{\pi^{\frac{d}{2}}}{\Gamma\left(\frac{d}{2}\right)}\, r^{\frac{d}{2}-1} \psi(r)+O_d\left(r^{\frac{d-3}{2}}+r^{\frac{d}{4}} \log r\right) \\
    &\qquad +  \frac{i\pi^{\frac{d}{2}}}{2\,\Gamma\left(\frac{d}{2}\right)} \sum_{2 \leq k \leq \sqrt{r}} \sideset{}{^*}\sum_{j=1}^{k-1} \frac{S_{j, k}}{k^d}  \frac{ e^{-i \pi \frac{j}{k}}}{\sin \frac{j}{k} \pi} e^{-2 \pi i \frac{j}{k}[r]} \, r^{\frac{d}{2}-1} .
    \end{aligned}
\end{equation}
\end{lemma}
\begin{proof}
The proof is just a further simplification of  \eqref{land1}, where the summation on $k$ will be divided into $k=1$ and $2\le k\le\sqrt{r}$. Then the first part contributes the main term  and the second  part will be incorporated into the error term. By the Euler-MacLaurin sum formula (see e.g.  \cite[Theorem 1.3]{Kra88}) we have
\begin{align}
    \sum_{0 \leq n \leq r} n^{\frac{d}{2}-1} & =\int_0^r u^{\frac{d}{2}-1} d u-r^{\frac{d}{2}-1} \psi(r)+\left(\frac{d}{2}-1\right) \int_0^r u^{\frac{d}{2}-2} \psi(u) d u \nonumber\\
    & =\frac{2}{d} r^{\frac{d}{2}}-r^{\frac{d}{2}-1} \psi(r)+O\left(r^{\frac{d}{2}-2}\right)  \label{lan22} ,
\end{align}
where we have used the second mean-value theorem for the last integral, as well as the following simple inequality 
\begin{equation*}
    \left|\int_{\xi_1}^{\xi_2} \psi(u) du \right| \le 1, \qquad\fall 0\le\xi_1<\xi_2<\infty,
\end{equation*}
which follows from the periodicity and mean-zero property of $\psi$.
On the other hand, the inner sum in  \eqref{land1}  can be handled by the  following lemma, whose proof shall be provided soon.
\begin{lemma} \label{lem21} 
Let $r\ge10$, $\az\ge1$ and $\lz\notin\zz$. Then
\begin{equation*}
    \sum_{0 \leq n \leq r} n^\alpha e^{-2 \pi i \lambda n} = \frac{i e^{-i \lambda\pi }}{2 \sin \lambda \pi} e^{-2 \pi i \lambda[r]} r^\alpha+O_\alpha\left(\frac{r^{\alpha-1}}{\|\lambda\|^2}\right),
\end{equation*}
with the notation $ \|u\|:=\min\limits_{n \in \zz}|u-n|, \, \forall \,  u\in \mathbb{R} .$
\end{lemma}
Recall that
\begin{equation} \label{ffaa} 
    \vol(B_d(1)) = \frac{\pi^\frac{d}{2}}{\Gz(\frac{d}{2}+1)}.
\end{equation}
Via  \eqref{lan22} and Lemma  \ref{lem21} (with $\az=d/2-1$ and $\lz=j/k$), together with the following estimate (by using \eqref{ff41})
\begin{equation*}
    \begin{aligned}
  \sum_{2 \leq k \leq \sqrt{r}} \sideset{}{^*}\sum_{j=1}^{k-1} \frac{|S_{j, k}|}{k^d} \frac{r^{\frac{d}{2}-2}}{\|\frac{j}{k}\|^2} &\lsim  r^{\frac{d}{2}-2}  \sum_{2 \leq k \leq \sqrt{r}} \sum_{j=1}^{{k-1}} k^{-\frac{d}{2}}\left(\frac{k^2}{j^2}+\frac{k^2}{(k-j)^2}\right) \\
    & \lsim r^{\frac{d}{2}-2} \sum_{2 \leq k \leq \sqrt{r}} k^{-\frac{d}{2}+2} \lesssim r^{\frac{d}{2}-\frac{3}{2}} ,
    \end{aligned}
\end{equation*}
we reach the desired conclusion  \eqref{ff61}.
\end{proof}

\begin{proof}[Proof of Lemma  \ref{lem21}]
We apply the unidimensional Poisson sum formula (see e.g.  \cite[(1.11)]{Kra88}) and obtain that 
\begin{align}
    \sum_{0 \leq n \leq r} n^\alpha e^{-2 \pi i \lambda n} & =\sum_{0 \leq n \leq[r]} n^\alpha e^{-2 \pi i \lambda n}=\frac{1}{2} [r]^\alpha e^{-2 \pi i \lambda[r]}+ \sideset{}{^{\prime \prime}}\sum_{0 \leq n \leq[r]} n^\alpha e^{-2 \pi i \lambda n} \nonumber\\
    & =\frac{1}{2}[r]^\alpha e^{-2 \pi i \lambda[r]}+\lim _{N \rightarrow \infty} \sum_{|n| \leq N} \int_0^{[r]} \rho^\alpha e^{-2 \pi i(\lambda+n) \rho} d \rho . \label{ff11} 
\end{align}
Through the integration by parts, the above integral can be assessed by 
\begin{align}
\int_0^{[r]} \rho^\alpha e^{-2 \pi i(\lambda+n) \rho} d \rho & =\left.\frac{\rho^\alpha e^{-2 \pi i(\lambda+n) \rho}}{-2 \pi i(\lambda+n)}\right|_0 ^{[r]}+\frac{\alpha}{2 \pi i(\lambda+n)} \int_0^{[r]} \rho^{\alpha-1} e^{-2 \pi i(\lambda+n) \rho} d \rho \nonumber\\
& =\frac{e^{-2 \pi i(\lambda+n)[r]}}{-2 \pi i(\lambda+n)}[r]^\alpha+\frac{\alpha}{2 \pi i(\lambda+n)} \int_0^{[r]} \rho^{\alpha-1} e^{-2 \pi i(\lambda+n) \rho} d \rho \nonumber\\
& =\frac{e^{-2 \pi i\lambda [r]}}{-2 \pi i(\lambda+n)}[r]^\alpha+ O_\az\left(\frac{r^{\alpha-1}}{(\lambda+n)^2}\right) .  \label{ff12}
\end{align}
On the other hand, invoking  \cite[6 of \S 1.445]{GR15} gives  
\begin{equation*}
    \begin{aligned}
 \lim _{N \rightarrow \infty} \sum_{|n| \leq N} \frac{1}{\lambda+n} 
     =\frac{1}{\lambda} - 2 \lambda \sum_{n=1}^{\infty} \frac{1}{n^2 - \lambda^2}
     =\frac{1}{\lambda}-2 \lambda\left(\frac{1}{2 \lambda^2}-\frac{\pi}{2} \frac{\cos \lambda\pi}{\lambda \sin \lambda \pi}\right) 
     =\frac{\pi \cos \lz\pi}{\sin \lambda \pi}.
    \end{aligned}
\end{equation*}
Notice that
\begin{equation*}
    [r]^\alpha=r^\alpha+ O_\alpha\left(r^{\alpha-1}\right), \qquad \frac{1}{|\sin \lambda \pi|} \sim \frac{1}{\|\lambda\|} \geq 2 .
\end{equation*}
Then by  \eqref{ff11}-\eqref{ff12}, we  get Lemma  \ref{lem21}. 
\end{proof}

\subsection{Estimates of the $\Psi$-sum}  \label{psiS} 
Let $\Psi$ be an $1$-periodic real-valued function on $\rr$ with  bounded variation $V(\Psi)$ on $[0,1]$. In 1983, S.B. Ste\u{c}kin  developed  the following technique  to transform the $\Psi$-sum into the exponential sum, whose detailed proof can be found in the book  \cite[\S 6.1.2]{Bor20}. Here we have just worked out the dependence for the upper bound $\mathbf{V}$ of $V(\Psi)$, where $\mathbf{V} $ is a positive constant.
\begin{lemma}[Ste\u{c}kin's inequality] \label{lem3w} 
Let $ f$ be a real-valued function,  $W\in\nn^*$ and $V(\Psi)\le \mathbf{V} $. Then for any integers $N_2>N_1>0$ and  set $S\subset \nn^*$, 
\begin{equation*}
    \sum_{\substack{N_1<n \leqslant N_2 \\ n \in S}} \Psi(f(n))=\sum_{\substack{N_1<n \leqslant N_2 \\ n \in S}} 1 \cdot \int_0^1 \Psi(t) d t + O_{\mathbf{V}}\left\{\frac{1}{W} \sum_{\substack{N_1<n \leqslant N_2 \\ n \in S}} 1+\sum_{m=1}^W \frac{1}{m}\left|\sum_{\substack{N_1<n \leqslant N_2 \\ n \in S}} e^{2\pi i m f(n)}\right|\right\} .
\end{equation*}
\end{lemma}

With Ste\u{c}kin's lemma, one can acquire an estimation of the $\Psi$-sum.
\begin{lemma} \label{lem4} 
Let $x\ge10$, $0\le A<B\le\sqrt{x}$ and $\Psi $ be given as above with 
$$\sup_{u\in[0,1]}|\Psi(u)| + V(\Psi)\le \overline{\bV} $$
for some constant $\overline{\bV}>0$.
If the integral of $\Psi$  over $[0,1]$  vanishes, then for any $L\ge1$ and $ R\in\rr$, there is a constant $C(\overline{\bV})>0$ depending only on $\overline{\bV}$ such that
\begin{equation} \label{well2} 
    \left|\sum_{A<n \leq B} \Psi\left(L^{-1}\sqrt{x-n^2}+R\right)\right| \le C(\overline{\bV}) \left( L^{-\frac{11}{41}} x^{\frac{497}{1506}} +L^{\frac{1}{2}} x^{\frac{1}{4}}\right).
\end{equation}
\end{lemma}
Before proceeding to the proof, we would like to explain how the problem of lattice point counting in Cygan--Kor\'anyi ball on Heisenberg groups connects the Gaussian circle problem in  essence, and how the index $\frac{497}{1506}$ in the right hand side of  \eqref{well2} arises.
\begin{remark} \label{reWel}  
(1) Recall that the lattice point discrepancy  \eqref{gauE} of the  Gauss circle problem has the following expression, for $x\ge10$,
\begin{equation} \label{gaup} 
    P_2(x)=-8 \sum_{\sqrt{x/2}<n\le \sqrt{x}} \psi(\sqrt{x-n^2}) +O(x^\frac14),
\end{equation}
cf. e.g.  \cite[(3.57)-(3.58)]{Kra88}. So the $\psi$-sum in  \eqref{gaup} is actually a particular case of that in  \eqref{well2}. Notice that it is widely conjectured that $P_2(x)= O_\ez(x^{\frac14+\ez})$ for any $\ez>0$. Hence one may ask whether the bound $L^{-\frac{11}{41}} x^{\frac{497}{1506}}$ could be improved to $L^{\az}x^{\frac14+\ez}$  with sufficiently small  $\az$. Once this holds, then for any $q\ge4$, our argument can lead to the better upper bound $|\cE(t)|\lsim_\ez t^{2q-1+\ez}$, which will furnish a verification of  Gath's conjecture. On the other hand, such $\Psi$-sum  occurs naturally in the error term of our Cygan--Kor\'anyi ball problem  (see  \eqref{Eqt}-\eqref{ff101}), whence it offers a new  evidence that this problem should be as challenging as the Gauss circle problem.

(2) In the paper of Nieland  \cite[]{Nie28} which handled the Gauss circle problem, the $ k$-th Derivative Test (cf. e.g. Lemma  \ref{ktest} below) due to van der Corput played an important role, where the optimal choice of $k$ is $5$. In our situation, the dual phases (derived from the B-process) of the trigonometric sums possess a similar structure (indeed, only an additional parameter is involved here); yet the summation ranges can be much larger, rendering the problem more complicated. Indeed, on account of the $5$-th Derivative Test, a rigid lower bound condition on the $5$-th derivative of the phase  is necessary. However, unlike the Gauss circle problem, the summation range in our case can be substantially larger, which results in a vanishing phenomenon for the $5$-th derivative (see Figure \ref{fig1} below for an illustration). To compensate for this deficiency, we select a pair of derivatives such that the non-vanishing interval of one covers the zeros of the other. It turns out that the optimal pair of orders of differentiation is $(5,6)$. The exponent deduced from the $6$-th derivative is exactly $\frac{497}{1506}$ appearing in \eqref{well2} for the case $q\ge4$, which is inferior to the exponent $\frac{27}{82}$  that could be  possibly obtainable from the $5$-th derivative. On the other hand, the parameter $L$  corresponds to the summation index $k $ in Landau's formula \eqref{ff61}; in the summing process, we lose some precision for $q=3$, compared with the higher-dimensional case. 
\end{remark}

\begin{proof}[Proof of Lemma  \ref{lem4}]
We are mainly devoted to the case where $B \leq x^{\frac{1}{2}}\left(1-\frac{1}{2} x^{-\frac{1}{4}}\right)^{\frac{1}{2}} $, and the rest situation is plain as shall be seen. Denote $\mathbb{N}_B:=\{n \in \mathbb{N}: n \leq B\} .$ For each integer $0 \leq j \leq \frac{\log x}{4 \log 2}$,
 we define the set 
\begin{equation*}
    U_j=U_j(B):=\left\{n \in \mathbb{N}_B:  2^{-j-1} x \leq x-n^2<2^{-j} x\right\}.
\end{equation*}
Note that   $U_j$ might be empty for large $j$, but in this case the estimate is trivial. Thus one  can assume that all $U_j$'s are nonempty henceforward.

Let $m\in\nn^*$ and define $ f_m(s):=mL^{-1} \sqrt{x-s^2} $, $s\in[0,\sqrt{x}]$. We first claim that, for any $0 \leq j \leq \frac{\log x}{4 \log 2}$,
\begin{equation} \label{l4w} 
    \begin{aligned}
    \sum_{n \in U_j} e^{2 \pi i f_m(n)} &=  e^{-\frac{\pi i}{4}} \sum_{n \in \widetilde{U}_j} g_m(n)\, e^{2 \pi i F_m(n)} \\
    &\qquad +O\left(2^{-\frac{3}{4} j} m^{-\frac{1}{2}} L^{\frac{1}{2}} x^{\frac{1}{4}} + \log \left(2^{\frac{j}{2}} m L^{-1}+2\right)\right) ,
    \end{aligned}
\end{equation}
where 
\begin{gather*}
\widetilde{U}_j:=\left\{n \in \mathbb{N}: m L^{-1}\left(2^j-1\right)^{\frac{1}{2}}<n \leq  \min\left\{m L^{-1}\left(2^{j+1}-1\right)^{\frac{1}{2}},|f'_m(B)|\right\}\right\} , \\[2mm]
    g_m(s)=m L^{-1} x^{\frac{1}{4}}\left(m^2 L^{-2}+s^2\right)^{-\frac{3}{4}} , \quad F_m(s):=x^{\frac{1}{2}}\left(m^2 L^{-2}+s^2\right)^{\frac{1}{2}} ,
\end{gather*}
and the implicit constant in the error term is absolute. To show  \eqref{l4w}, we use the B-process of van der Corput. The following version can be seen from  \cite[p. 210]{IK04}.
\begin{lemma}[B-process]  \label{lemVB} 
Let $f\in C^4([a,b])$ be a real-valued function. If there exist constants $\Lz>0$, $\eta\ge1$ such that
\begin{equation*}
    \Lambda \leq\left|f^{\prime \prime}(s)\right| \leq \eta \Lz, \quad\left|f^{(3)}(s)\right| \leq \eta \Lz(b-a)^{-1}, \quad\left|f^{(4)}(s)\right| \leq \eta \Lz(b-a)^{-2},
\end{equation*}
then one has
\begin{equation*}
    \begin{aligned}
   \sum_{a<k\le b} e^{2 \pi i f(k)} & =e^{\frac{\pi i}{4} \sgn f^{\prime\prime}\left(\frac{a+b}{2}\right)} \sum_{\left.n \in f^{\prime}((a, b])\right)} e^{2 \pi i \left(f\left(s_n\right)-n s_n\right)}\left|f^{\prime \prime}\left(s_n\right)\right|^{-\frac{1}{2}} \\
    &\qquad +  O\left(\Lambda^{-\frac{1}{2}}+\eta^2 \log \left(\left|f^{\prime}(b)-f^{\prime}(a)\right|+2\right) \right).
    \end{aligned}
\end{equation*}
Here $s_n$ is the unique solution of the equation $f'(s)=n$, and the implicit constant in the error term is absolute.
\end{lemma}

Observe that
\begin{equation*}
   U_j= \left(x^{\frac{1}{2}}\left(1-2^{-j}\right)^{\frac{1}{2}}, \min\left\{ x^{\frac{1}{2}}\left(1-2^{-j-1}\right)^{\frac{1}{2}},\, B\right\}\right] \cap \nn.
\end{equation*}
A trivial calculation gives 
\begin{equation*}
    \begin{array}{ll}
    f_m^{\prime}(s)=- mL^{-1} s\left(x-s^2\right)^{-\frac{1}{2}}, & f_m^{\prime \prime}(s)=- mL^{-1} x\left(x-s^2\right)^{-\frac{3}{2}}, \\[2mm]
    f_m^{(3)}(s)=-3 mL^{-1} x s\left(x-s^2\right)^{-\frac{5}{2}}, & f_m^{(4)}(s)=-3 mL^{-1} x\left(x+4 s^2\right)\left(x-s^2\right)^{-\frac{7}{2}} .
    \end{array}                
\end{equation*}
Then it is routine to check all conditions in Lemma  \ref{lemVB} are fulfilled  with $\Lambda=2^{\frac{3}{2} j} mL^{-1} x^{-\frac{1}{2}} $ and certain $\eta \geq 2^{\frac{3}{2}}$ being sufficiently large, from which and the fact that both  $g_m$ and $F_m$ are even functions, the forgoing claim follows, after some simple estimations.

Next we dominate the second term in  \eqref{l4w}. An argument of summation by parts will reduce the problem to  estimating exponential sums of the type $\sum_{n \in \widehat{U}_j}  e^{2 \pi i F_m(n)}$ with $\widehat{U}_j$ being some subset of ${U}_j$ (see  \eqref{set1} and  \eqref{set2} below). To this end, let us further assume that 
\begin{equation} \label{assu} 
    2^{\frac{j}{2}}mL^{-1}\ge1,\qquad \fall 0\le j \leq \frac{\log x}{4 \log 2}.
\end{equation}
In fact, if this does not hold for  $j_0$, then one has
$ 2^{\frac{j_0}{2}}mL^{-1}<1$, which implies that $$\#\,\widetilde{U}_{j_0} \le \left(m L^{-1}\left(2^{j_0+1}-1\right)^{\frac{1}{2}} - m L^{-1}\left(2^{j_0} -1\right)^{\frac{1}{2}}\right) \sim 2^{\frac{j_0}{2}}mL^{-1} \le 1.$$
Since on each $\widetilde{U}_j$,
\begin{equation} \label{gmE} 
    g_m(n) \sim 2^{-\frac{3}{4} j} m^{-\frac{1}{2}} L^{\frac{1}{2}} x^{\frac{1}{4}},
\end{equation}
then for $j_0$ the second term in  \eqref{l4w} can be absorbed by the error term and we are done. 

Another direct computation shows that the  first six derivatives of $x^{-\frac12}F_m$ are:
\begin{gather*}
    \frac{s}{\sqrt{m_L^2+s^2}},\quad\frac{m_L^2}{\left(m_L^2+s^2\right)^{3/2}}, \quad -\frac{3 m_L^2 s}{\left(m_L^2+s^2\right)^{5/2}}, \quad -\frac{3 m_L^2 \left(m_L^2-4 s^2\right)}{\left(m_L^2+s^2\right)^{7/2}}, \\[2mm]
    \frac{15 \left(3 m_L^4 s-4 m_L^2 s^3\right)}{\left(m_L^2+s^2\right)^{9/2}}, \quad \frac{45 \left(m_L^6-12 m_L^4 s^2+8 m_L^2 s^4\right)}{\left(m_L^2+s^2\right)^{11/2}},
\end{gather*}
with $m_L:=m L^{-1}$. Figure  \ref{fig1} depicts the graph for $5,6$-th derivatives of  $ x^{-\frac12}F_m$ in the case where $m=L$.
\begin{figure}[htbp]      
    \centering         
    \includegraphics[width=0.7\textwidth]{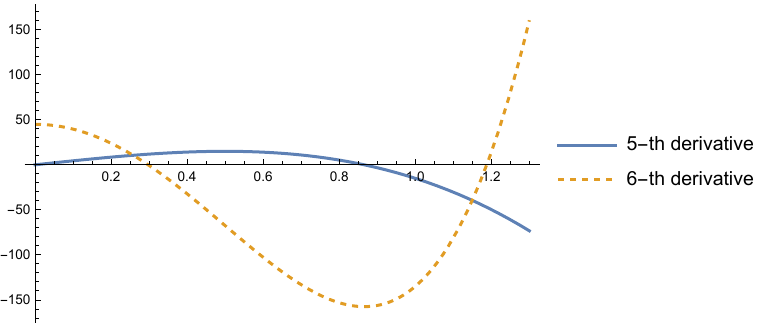}  
    \caption{the $5,6$-th derivatives of $x^{-\frac12}F_m$ (with $m=L$)} \label{fig1}    
\end{figure}

As mentioned above, it is also necessary to distinguish two cases: $j\ge1$ and $j=0$, where the  $5$-th test is best possible for the first case, but the second requires both the $5$-th and $6$-th tests. When $1\le j \leq \frac{\log x}{4 \log 2} $, for any integer $N$ obeying 
\begin{equation} \label{ff65} 
    m L^{-1}\left(2^j-1\right)^{\frac{1}{2}}<N \leq \min \left\{m L^{-1}\left(2^{j+1}-1\right)^{\frac{1}{2}},\left|f_m^{\prime}(B)\right|\right\}, 
\end{equation}
we  introduce 
\begin{equation} \label{set1} 
    \widetilde{U}_{j, N}:=\left\{n \in \mathbb{N}: m L^{-1}\left(2^j-1\right)^{\frac{1}{2}}<n \leq N\right\}.
\end{equation}
While for $j=0$, note that $n\in\widetilde{U}_0$ is equivalent to  $ 0<\frac{n}{m L^{-1}} \leq \min \left\{1,\left|f_m^{\prime}(B)\right| m^{-1} L\right\}$. Given any $\az>0$ we shall write 
\begin{equation*}
    (\az)_B:=\min \left\{\az,\left|f_m^{\prime}(B)\right| m^{-1} L\right\}.
\end{equation*}
And then we decompose  $ \widetilde{U}_0 $ in the form
\begin{equation*}
    \begin{aligned}
    \widetilde{U}_0 & = \left\{n \in \mathbb{N}: \frac{n}{m L^{-1}} \le\left(\frac{1}{4}\right)_B\right\} \bigcup\left\{n \in \mathbb{N}: \frac{1}{4}<\frac{n}{m L^{-1}}<\left(\frac{1}{2}\right)_B\right\} \\
    &\qquad \bigcup\left\{n \in \mathbb{N}: \frac{1}{2} \le \frac{n}{m L^{-1}} \le(1)_B\right\}=:\bigcup_{l=1}^3 \widetilde{U}_0^{(l)}.
    \end{aligned}
\end{equation*}
Correspondingly, for any integer $N_j\ (j=1,2,3)$ subject to 
\begin{equation} \label{ff651} 
    0<N_1 \le\left(\frac{1}{4}\right)_B mL^{-1}<N_2 \le\left(\frac{1}{2}\right)_B mL^{-1}<N_3 \le(1)_B \, mL^{-1},
\end{equation}
we denote
\begin{equation}\label{set2} 
\begin{gathered} 
     \widetilde{U}_{0, N_1}^{(1)}:=\left\{n \in \mathbb{N}: n \leq N_1\right\}, \quad \widetilde{U}_{0, N_2}^{(2)}:=\left\{n \in \mathbb{N}: 4^{-1} mL^{-1} < n \leq N_2\right\}, \\[2mm]
     \widetilde{U}_{0, N_3}^{(3)}:=\left\{n \in \mathbb{N}: 2^{-1}mL^{-1} \leq n \leq N_3\right\} . 
\end{gathered}
\end{equation}

To bound the trigonometric sums over the sets in \eqref{set1} and \eqref{set2}, we invoke the following $k$-th Derivative Test due to van der Corput (see e.g.  \cite[Theorem 8.20]{IK04}).
\begin{lemma}[$k$-th Derivative Test]   \label{ktest} 
Let $k\ge2$. Suppose that $b-a\ge1$, and that $f$ is a real-valued function defined on $(a,b)$ satisfying $\Lz\le |f^{(k)}(s)| \le \eta \Lz$ for some $\Lz>0$ and $\eta\ge1$. Then there is an absolute constant $C$ such that
\begin{equation*}
    \left|\sum_{a<n<b} e^{2\pi i f(n)} \right| \le C \left(\eta^{2^{2-k}} \Lambda^\kappa(b-a)+\Lambda^{-\kappa}(b-a)^{1-2^{2-k}} \right),
\end{equation*}
where $\kappa=\left(2^k -  2\right)^{-1}$.
\end{lemma}

As $j\ge1$, we can check that
\begin{equation} \label{chk1} 
    \left|F_m^{(5)}(n)\right| \sim x^{\frac{1}{2}} m^2 L^{-2}\left(\left(2^j-1\right)^{\frac{1}{2}} m L^{-1}\right)^{-6} \sim  2^{-3 j} m^{-4} L^4 x^{\frac{1}{2}}  
\end{equation}
uniformly on $ \widetilde{U}_{j, N}$, and the length of the interval fulfills  $b-a \lsim 2^{\frac{j}2}mL^{-1}$. Then via the $5$-th Derivative Test along with the assumption  \eqref{assu},  we obtain that  
\begin{equation} \label{ff71} 
    \left|\sum_{n \in \widetilde{U}_{j, N}} e^{2 \pi i F_m(n)}\right| \lsim 2^{\frac{2}{5} j}\left(m L^{-1}\right)^{\frac{13}{15}}  x^{\frac{1}{60}}+2^{\frac{43}{80} j}\left(m L^{-1}\right)^{\frac{121}{120}} x^{-\frac{1}{60}}.
\end{equation}

As $j=0$, the interval length  of every $\widetilde{U}_{0, N_l}^{(l)}$ ($l=1,2,3$) can be majorized by $C\,mL^{-1}$ for some absolute constant $C$. Moreover, we can justify that  \eqref{chk1} still holds uniformly on $\widetilde{U}_{0, N_2}^{(2)}$; and for $j=1,3$ we use the $6$-th Derivative Test instead, with the following
\begin{equation*}
   \left|F_m^{(6)}(n)\right| \sim    m^{-5} L^5 x^{\frac{1}{2}}, \quad \mbox{on}\ \ \widetilde{U}_{0, N_1}^{(1)} \cup \widetilde{U}_{0, N_3}^{(3)}.
\end{equation*}
In this way one can get that, under assumption  \eqref{assu},
\begin{equation} \label{ff72} 
    \begin{gathered}
    \left|\sum_{n \in \widetilde{U}_{0, N_2}^{(2)}} e^{2 \pi i F_m(n)}\right| \lesssim\left(m L^{-1}\right)^{\frac{13}{15}} x^{\frac{1}{60}}+\left(m L^{-1}\right)^{\frac{121}{120}} x^{-\frac{1}{60}} , \\
    \left|\sum_{n \in \widetilde{U}_{0, N_1}^{(1)}} e^{2 \pi i F_m(n)}\right| + \left|\sum_{n \in \widetilde{U}_{0, N_3}^{(3)}} e^{2 \pi i F_m(n)}\right| \lesssim\left(m L^{-1}\right)^{\frac{57}{62}} x^{\frac{1}{124}}+\left(m L^{-1}\right)^{\frac{505}{496}} x^{-\frac{1}{124}}.   
    \end{gathered}
\end{equation}
As a result, from partial summation and the monotonicity of $g_m$, together with  \eqref{ff71} and  \eqref{gmE}, we conclude that, for $1  \le j \leq \frac{\log x}{4 \log 2}$, 
\begin{equation*}
    \begin{aligned}
    \left|\sum_{n \in \widetilde{U}_j} g_m(n) e^{2 \pi i F_m(n)}\right| & \lesssim g_m\left(m L^{-1}\left(2^j-1\right)^{\frac{1}{2}}\right) \cdot \max _{N \ \begin{footnotesize}
    \mbox{satisfying  \eqref{ff65}}
    \end{footnotesize}}\left|\sum_{n \in \widetilde{U}_{j, N}} e^{2 n i F_m(n)}\right| \\[2mm]
    & \lesssim 2^{-\frac{7}{20} j}\left(m L^{-1}\right)^{\frac{11}{30}} x^{\frac{4}{15}}+2^{-\frac{17}{80} j}\left(m L^{-1}\right)^{\frac{61}{120}} x^{\frac{7}{30}} .
    \end{aligned}
\end{equation*}
As to $j=0$, similarly for $\widetilde{U}_0^{(2)}$ and using  \eqref{ff72} for $\widetilde{U}_0^{(l)}\, (l=1,3)$ in place of  \eqref{ff71}, we can deduce that
\begin{gather}
    \left|\sum_{n \in \widetilde{U}_0^{(2)}} g_m(n) e^{2 \pi i F_m(n)}\right| \lesssim\left(m L^{-1}\right)^{\frac{11}{30}} x^{\frac{4}{15}}+\left(m L^{-1}\right)^{\frac{61}{120}} x^{\frac{7}{30}} , \nonumber\\
    \left|\sum_{n \in \widetilde{U}_0^{(1)}} g_m(n) e^{2\pi i F_m(n)}\right|+\left|\sum_{n \in \widetilde{U}_0^{(3)}} g_m(n) e^{2 \pi i F_m(n)}\right| 
     \lesssim\left(m L^{-1}\right)^{\frac{13}{31}} x^{\frac{8}{31}}+\left(m L^{-1}\right)^{\frac{257}{496}} x^{\frac{15}{62}} . \label{ff91} 
\end{gather}
Consequently, from   \eqref{l4w} it follows that, whenever $1  \le j \leq \frac{\log x}{4 \log 2}$,
\begin{equation} \label{ff73} 
    \begin{aligned}
    \left|\sum_{n \in U_j} e^{2 \pi i f_m(n)}\right|&  \lesssim  2^{-\frac{7}{20 } j}\left(m L^{-1}\right)^{\frac{11}{30}} x^{\frac{4}{15}}+2^{-\frac{17}{80} j}\left(m L^{-1}\right)^{\frac{61}{120}} x^{\frac{7}{30}} \\
    &\qquad +2^{-\frac{3}{4} j} m^{-\frac{1}{2}} L^{\frac{1}{2}} x^{\frac{1}{4}}+\log \left(2^{\frac{j}{2}} m L^{-1}+2\right).
    \end{aligned}
\end{equation}

Let $W\ge2$ be chosen. Through  \eqref{ff73}, an application of Lemma  \ref{lem3w}  leads to that, for $1  \le j \leq \frac{\log x}{4 \log 2}$, 
\begin{equation*} \label{ff81} 
    \begin{aligned}
    \left|\sum_{n \in U_j} \Psi\left(L^{-1}\sqrt{x-n^2}+R\right)\right| & \lesssim_{\overline{\bV}} 2^{-j} W^{-1} x^{\frac{1}{2}}+\sum_{1 \leq m \leq W} \frac{1}{m}\left|\sum_{n \in U_j} e^{2 \pi i f_m{(n)}}\right| \\
    & \lesssim 2^{-j} W^{-1} x^{\frac{1}{2}}+2^{-\frac{7}{20 }j}\left(W L^{-1}\right)^{\frac{11}{30}} x^{\frac{4}{15}}+2^{-\frac{17}{80} j}\left(W L^{-1}\right)^{\frac{61}{120}} x^{\frac{7}{30}} \\[2mm]
    &\qquad +2^{-\frac{3}{4} j} L^{\frac{1}{2}} x^{\frac{1}{4}}+j+\log ^2 W.
    \end{aligned}
\end{equation*}
We balance the first two terms in  the right hand side of the above inequality, which decides that $W=2^{-\frac{39}{82} j} x^{\frac{7}{41}} L^{\frac{11}{41}},$ thereby showing  
\begin{equation} \label{ff82} 
    \begin{aligned}
    \left|\sum_{n \in U_j} \Psi\left(L^{-1}\sqrt{x-n^2}+R\right)\right| &\lesssim_{\overline{\bV}} 2^{-\frac{43}{82} j} L^{-\frac{11}{41}} x^{\frac{27}{82}}+2^{-\frac{149}{328 } j} L^{-\frac{61}{164}} x^{\frac{105}{328}} \\[-4mm]
    &\qquad +2^{-\frac{3}{4} j} L^{\frac{1}{2}} x^{\frac{1}{4}}+j+\log ^2(L x) \\[2mm]
    & \le 2^{-\frac{149}{328} j} L^{-\frac{11}{41}} x^{\frac{27}{82}}+2^{-\frac{3}{4} j} L^{\frac{1}{2}} x^{\frac{1}{4}}+j+\log ^2(L x) .
    \end{aligned}
\end{equation}
While  for $j=0$ we have   
\begin{equation} \label{ff83} 
\left|\sum_{n \in U_0} \Psi\left(L^{-1}\sqrt{x-n^2}+R\right)\right| \lesssim_{\overline{\bV}}  L^{-\frac{208}{753}}  x^\frac{497}{1506}+L^{\frac{1}{2}} x^{\frac{1}{4}}.  
\end{equation}
As a matter of fact, it is easy to see that \eqref{ff82} is still  valid for $j=0$  with $ U_j$ on the left hand side replaced by $U_0^{(2)}$; and likewise, on account of  Lemma \ref{lem3w} and  \eqref{l4w} and \eqref{ff91}, it holds
\begin{equation*}
    \begin{aligned}
    &\left|\sum_{n \in U_0^{(1)}} \Psi\left(L^{-1}\sqrt{x-n^2}+R\right)\right| + \left|\sum_{n \in  U_0^{(3)}} \Psi\left(L^{-1}\sqrt{x-n^2}+R\right)\right| \\
    &\qquad \lesssim_{\overline{\bV}}  W^{-1} x^{\frac{1}{2}}+\left(W L^{-1}\right)^{\frac{13}{31}} x^{\frac{8}{31}}+\left(W L^{-1}\right)^{\frac{257}{496}} x^{\frac{15}{62}}     +L^{\frac{1}{2}} x^{\frac{1}{4}}.
    \end{aligned}
\end{equation*}
Therefore, a balance between the first and the third terms produces   $W=L^{\frac{257}{753}} x^{\frac{128}{753}}$. This forces that   the left hand side of the above inequality can be controlled by
\begin{equation*}
 L^{-\frac{257}{753}} x^{\frac{497}{1506}}+L^{-\frac{208}{753}} x^{\frac{248}{753}}  +L^{\frac{1}{2}} x^{\frac{1}{4}}+\log ^2(L x) \lsim  L^{-\frac{208}{753}} x^{\frac{497}{1506}}+L^{\frac{1}{2}} x^{\frac{1}{4}},
\end{equation*}
with the implicit constant depending only on $ \overline{\bV}$. This proves  \eqref{ff83}.

We are in the position to demonstrate  \eqref{well2}. If $B \leq x^{\frac{1}{2}}\left(1-\frac{1}{2} x^{-\frac{1}{4}}\right)^{\frac{1}{2}} $, then it yields from  \eqref{ff82}-\eqref{ff83} that
\begin{equation*}
    \begin{aligned}
    \left|\sum_{A<n \leq B} \Psi\left(L^{-1}\sqrt{x-n^2}+R\right)\right| & \le\left|\sum_{0<n \leq B} \Psi\left(L^{-1}\sqrt{x-n^2}+R\right)\right|+\left|\sum_{0<n \leq A} \cdots\right| \\
    & \le \sum_{0 \leq j \leq \frac{\log x}{4 \log 2}}\left(\left|\sum_{n \in U_j(A)} \cdots\right|+  \left|\sum_{n \in U_j(B)} \cdots\right| \right) \\
    & \lesssim_{\overline{\bV}} L^{-\frac{11}{41}} x^{\frac{497}{1506}}+L^{\frac{1}{2}} x^{\frac{1}{4}},
    \end{aligned}
\end{equation*}
as required. On the contrary, if $B > x^{\frac{1}{2}}\left(1-\frac{1}{2} x^{-\frac{1}{4}}\right)^{\frac{1}{2}}=:\widetilde{B}$, we split 
\begin{equation*}
    (A,B]\cap\nn = ((A,\widetilde{B}]\cap\nn )\cup ((\widetilde{B},B]\cap\nn).
\end{equation*}
Thus the summation on the first range can be majorized as in the former estimation; while for the second one, the summation over it can be seen to be bounded by 
\begin{equation*}
  \left|\sum_{\widetilde{B}<n \le B} \Psi\left(L^{-1}\sqrt{x-n^2}+R\right)\right| \le  \sum_{\widetilde{B}<n \le B}  \overline{\bV} \le \overline{\bV} (B-\widetilde{B}+1) \lsim_{\overline{\bV}} x^{\frac{1}{4}}.
\end{equation*}
From these, the proof of  Lemma  \ref{lem4} is  completed.
\end{proof}
\subsection{Proof of Theorem   \ref{mainT}}  \label{prft} 
Let $q\ge3$. In what follows,   the letters $m$ and $n$ will be reserved for integer vectors in the summation, whose dimensions are evident from the text.

The defining inequality for $\mathcal B_t$ gives
\begin{align*}
    \#\left(\mathbb{Z}^{2 q+1} \cap \dz_t\, \cB_1\right) &= \sum_{ \left(\sum_{j=1}^{2q} m_j^2\right)^2+n^2 \leq t^4} 1 \\
    &= \sum_{|n|\le t^2}\,\sum_{\sum_{j=1}^{2q} m_j^2\le \sqrt{t^4-n^2}} 1 \\
    &=  \sum_{|n|\le t^2} \vol(B_{2q}(1)) (t^4-n^2)^{\frac{q}2} + \sum_{|n|\le t^2} P_{2q}(\sqrt{t^4-n^2}).
\end{align*}
From  \cite[Lemma 3.12]{Kra88} (with $k=2,\,\nu =\frac{q+1}{2}$ and $x=t^2$ therein) and  \eqref{ffaa}, the first resulting term equals
\begin{equation*}
    \frac{\Gamma\left(\frac{q}{2}+1\right)}{\Gamma\left(\frac{q}{2}+\frac{3}{2}\right)} \frac{\pi^{q+\frac{1}{2}}}{\Gamma(q+1)} t^{2 q+2}+O\left(t^q\right) .
\end{equation*}
Note that
\begin{equation*}
    \begin{aligned}
    \operatorname{vol}\left(\cB_1\right) & =2 \int_0^1 \int_{|v| \le\left(1-w^2\right)^{\frac{1}{4}}} d v d w=\frac{2 \pi^q}{\Gamma(q+1)} \int_0^1\left(1-w^2\right)^{\frac{q}{2}} d w \\
    & =\frac{\pi^q}{\Gamma(q+1)} \rB\left(\frac{q}{2}+1, \frac{1}{2}\right)=\frac{\pi^{q+\frac{1}{2}} \Gamma\left(\frac{q}{2}+1\right)}{\Gamma(q+1) \Gamma\left(\frac{q}{2}+\frac{3}{2}\right)},
    \end{aligned}
\end{equation*}
where $\rB(\cdot,\cdot)$ is the usual Beta function. Hence 
\begin{equation} \label{Eqt} 
    \cE_q(t)=  \sum_{|n|\le t^2} P_{2q}(\sqrt{t^4-n^2}) + O(t^q).
\end{equation}

Put $t=x^\frac14$. By Landau's formula  \eqref{ff61} (with $d=2q$) we have 
\begin{equation} \label{ff77} 
    \begin{aligned}
    \sum_{|n| \leq \sqrt{x}} P_{2 q}\left(\sqrt{x-n^2}\right)= & -\frac{2 \pi^{q}}{\Gamma\left(q\right)} I(x)+O\left(x^{\frac{{2q}-1}{4}}+x^{\frac{q+2}{4}} \log x\right) \\
    & + \frac{i \pi^{q}}{\Gamma\left(q\right)} \sum_{2 \leq k \leq x^\frac14} \sideset{}{^*}\sum_{j=1}^{k-1} \frac{S_{j, k}}{k^{2q}}  \frac{  e^{-i \pi \frac{j}{k}}}{\sin \frac{j}{k} \pi}  T_{j,k}(x),
    \end{aligned}
\end{equation}
with
\begin{gather*}
    I(x):=\sum_{0 < n \leq \sqrt{x}}\left(x-n^2\right)^{\frac{q-1}{2}} \psi\left(\sqrt{x-n^2}\right),\\
    T_{j, k}(x) :=\sum_{0<n\le\sqrt{x-k^4}}\left(x-n^2\right)^{\frac{q-1}{2}} e^{-2 \pi i \frac{j}{k}\left[\sqrt{x-n^2}\right]} .
\end{gather*}
Here we have also used the simple inequality
\begin{equation*}
    x^{\frac{q-1}{2}} \sum_{2 \leq k \leq x^{\frac{1}{4}}}   \sum_{j=1}^{k-1} k^{-q}  \left(\frac{k}{j}+\frac{k}{k-j}\right) \lsim x^{\frac{q-1}{2}} \sum_{2 \leq k \leq x^{\frac{1}{4}}} k^{-q+1} \log k \lsim x^{\frac{{2q}-1}{4}}.
\end{equation*}

Via  \cite[Theorem 1.2]{Kra88} we obtain
\begin{align}  
    I(x)  & =(q-1) \int_0^{\sqrt{x}} u\left(x-u^2\right)^{\frac{q-3}{2}} \sum_{0<n \le u} \psi\left(\sqrt{x-n^2}\right) d u \nonumber\\
    & = (q-1) x^{\frac{q-1}{2}} \int_0^1 u\left(1-u^2\right)^{\frac{q-3}{2}} \sum_{0<n \le u \sqrt{x}} \psi\left(\sqrt{x-n^2}\right) d u . \label{ff103}
\end{align}
and
\begin{align}
T_{j,k}(x)
={}&k^{2q-2}A_{j,k}(\sqrt{x-k^4};x)\nonumber\\
&+(q-1)\, x^{\frac{q-1}{2}}
\int_0^{\sqrt{1-\frac{k^4}{x}}}u(1-u^2)^{\frac{q-3}{2}}
 A_{j,k}(u\sqrt x;x)\,du,
\label{ff102}
\end{align}
where   for $0<\rho\le\sqrt{x}$, we have written
\[
A_{j,k}(\rho;x)
:=\sum_{0<n\le \rho}e^{-2\pi i \frac{j}{k}[\sqrt{x-n^2}]}.
\]

Now we introduce a family of $1$-periodic functions $\{\psi_k\}_{k\ge2}$ as follows:
\begin{equation*}
    \psi_k(u) := \begin{cases}1-\frac{1}{k}, & \text { if }\{u\}<\frac{1}{k}, \\ -\frac{1}{k}, & \text { if } \frac{1}{k} \leq\{u\}<1 .\end{cases}
\end{equation*}
Notice that 
\begin{equation*}
    \begin{aligned}
     A_{j,k}(\rho;x)&=\sum_{s=0}^{k-1} e^{-2 \pi i \frac{j}{k} s} \sum_{0 <n \leq \rho, \,\left[\sqrt{x-n^2}\right] \equiv s \bmod k} 1 \\
    & =\sum_{s=0}^{k-1} e^{-2 \pi i \frac{j}{k} s} \sum_{0<n\le \rho,\,\{k^{-1}(x-n^2-s)\}< k^{-1}} 1  \\
    & =\sum_{s=0}^{k-1} e^{-2 \pi i \frac{j}{k} s} \sum_{0<n \leq \rho}{\frac{1}{k}} +\sum_{s=0}^{k-1} e^{-2 \pi i \frac{j}{k} s} \sum_{0<n\le \rho} \psi_k\left(\frac{\sqrt{x-n^2}-s}{k}\right) \\
    & =\sum_{s=0}^{k-1} e^{-2 \pi i \frac{j}{k} s} \sum_{0<n\le \rho} \psi_k\left(\frac{\sqrt{x-n^2}-s}{k}\right) . \\
    \end{aligned}
\end{equation*}
Thus, for $2 \le k\le x^{\frac14}$ and $0<\rho\le\sqrt{x}$,
\begin{equation} \label{ff101} 
    \left| A_{j,k}(\rho;x)\right| \le \min \left\{x^{\frac{1}{2}}, k \cdot \max _{0 \leq s \leq k-1}\left|\sum_{0<n\le \rho} \psi_k\left(\frac{\sqrt{x-n^2}}{k}-\frac{s}{k}\right)\right|\right\}.
\end{equation}

Owing to  \eqref{ff103} and Lemma  \ref{lem4} with ($L=1,\, R=0,\, B=u\sqrt{x}$ and $A=0$) we find that
\begin{equation*}
    |I(x)| \lesssim x^{\frac{q-1}{2}} \cdot x^{\frac{497}{1506}} \int_0^1 u\left(1-u^2\right)^{\frac{q-3}{2}} d u \lesssim x^{\frac{q-1}{2}+\frac{497}{1506}} .
\end{equation*}
On the other hand, it is easily seen that for any $k\ge2$, $\psi_k$ has vanishing integral over $[0,1]$ with 
$$\sup_{u\in[0,1]} |\psi_k(u)| + V(\psi_k) \le 10, $$
whence it follows from  \eqref{ff102}-\eqref{ff101} and Lemma  \ref{lem4} (with $L=k$ and $R=s/k$) that
\begin{equation*}
    \begin{aligned}
    \left|T_{j, k}(x)\right| & \lesssim (k^{{2q}-2}+ x^{\frac{q-1}{2}} ) \min \left\{x^{\frac{1}{2}}, k \cdot\left(k^{-\frac{11}{41}} x^{\frac{497}{1506}}+k^{\frac{1}{2}} x^{\frac{1}{4}}\right)\right\} \\[2mm]
    & \lesssim x^{\frac{q-1}{2}}\left(k^{\frac{30}{41}} x^{\frac{497}{1506}}+\min \left\{x^{\frac{1}{2}}, k^{\frac{3}{2}} x^{\frac{1}{4}}\right\} \right),
    \end{aligned}
\end{equation*}
where we have used $2\le k\le x^{\frac14}$ in the last line. Consequently, by  \eqref{ff77} we arrive at that
\begin{equation*}
    \begin{aligned}
     &\left|\sum_{|n| \leq \sqrt{x}} P_{2 q}\left(\sqrt{x-n^2}\right)\right| \lesssim x^{\frac{q-1}{2}+\frac{497}{1506}}+x^{\frac{2q-1}{4}}+x^{\frac{q+2}{4}} \log x \\
    &\qquad +\sum_{2 \leq k \leq x^{\frac{1}{4}}} \sum_{j=1}^{k-1} k^{-q}\left(1+\left(\frac{k}{j}+\frac{k}{k-j}\right)\right) x^{\frac{q-1}{2}}\left(k^{\frac{30}{41}} x^{\frac{497}{1506}}+\min \left\{x^{\frac{1}{2}}, k^{\frac{3}{2}} x^{\frac{1}{4}}\right\}\right) \\
    & \qquad \lesssim \begin{cases}x^{\frac{q-1}{2}+\frac{497}{1506}}, & q \geq 4, \\
    x^{\frac{4}{3}} \log x, & q=3.\end{cases}
    \end{aligned}
\end{equation*}
This, together with  \eqref{Eqt}, implies Theorem  \ref{mainT}.

\section{Funding}
Sheng-Chen Mao is   supported by the China Postdoctoral Science Foundation.  Sibei Yang is partially supported by the National Natural Science Foundation of China
(Grant No. 12431006), Longyuan Young Talents of Gansu Province, and the Key Project of Gansu Provincial National
Science Foundation (Grant No. 23JRRA1022).

\vspace{.5cm}
\begin{small}
\noindent{\bf Data Availability}\quad Not applicable. No datasets were generated or analysed during the current study.
\end{small}

\section*{Declarations}
\begin{small}
\noindent{\bf Conflict of interest}\quad The authors declare  that there is no conflict of interest.
\end{small}

\phantomsection\addcontentsline{toc}{section}{References}
\bibliographystyle{abbrv}
\bibliography{LPCP2606}

\vspace{2cm}
\noindent Sheng-Chen Mao ({\it Corresponding author}) and Sibei Yang    

\medskip\noindent
School of Mathematics and Statistics, Gansu Key Laboratory of Applied Mathematics and Complex Systems, Lanzhou University, Lanzhou 730000, The People’s Republic of China    

\medskip\noindent
\begin{tabular}{@{}ll@{}}
{\it E-Mails:}&{\ttfamily maoshengchen@lzu.edu.cn; maosci@163.com } (S.-C. Mao) \\
&{\ttfamily yangsb@lzu.edu.cn } (S. Yang)
\end{tabular}

\end{document}